\documentclass[10pt]{amsart}
\usepackage{amsmath}
\usepackage{amsfonts}
\usepackage{amssymb}
\usepackage{amsthm}
\usepackage{url}
\usepackage{tikz-cd}
\usepackage{dsfont}
\usepackage{graphicx}
\usepackage{caption}
\usepackage{subcaption}
\usepackage{hyperref}
\usepackage{color}
\usepackage{enumerate}
\usepackage{textcomp}

\usepackage[all]{xy}

\numberwithin{equation}{section}

\newtheorem{proposition}{Proposition}[section]

\newtheorem{theorem}[proposition]{Theorem}
\newtheorem{corollary}[proposition]{Corollary}

\theoremstyle{definition}
\newtheorem{remark}[proposition]{Remark}

\DeclareMathOperator{\Bl}{Bl}

\DeclareMathOperator{\GL}{GL}
\DeclareMathOperator{\Aut}{Aut}

\DeclareMathOperator{\Gr}{Gr}

\def\R{\mathbb{R}}

\def\Z{\mathbb{Z}}

\def\P{\mathbb{P}}

\def\C{\mathbb{C}}

\def\cE{\mathcal{E}}
\def\cO{\mathcal{O}}

\def\cN{\mathcal{N}}

\def\cS{\mathcal{S}}

\def\cK{\mathcal{K}}

\def\cB{\mathcal{B}}

\def\cC{\mathcal{C}}

\def\cS{\mathcal{S}}

\def\Gr{\mathrm{Gr}}

\def\delb{\overline{\partial}}
\def\del{\partial}

\def\Aut{\mathrm{Aut}}
\def\Ad{\mathrm{Ad}}
\def\Fut{\mathrm{Fut}}
\def\PGL{\mathrm{PGL}}

\def\aut{\mathfrak{aut}}

\def\sl{\mathfrak{sl}}
\def\gl{\mathfrak{gl}}

\def\om{\omega}
\def\ep{\varepsilon}

\def\>{\rangle}

\def\<{\langle}
\def\>{\rangle}

\title[The Futaki invariant of Fano threefolds]{On the Futaki invariant of Fano threefolds}

\author[Lars Martin Sektnan and Carl Tipler]{Lars Martin Sektnan and Carl Tipler}

\address{Lars Martin Sektnan, Department of Mathematical Sciences, University of Gothenburg, 412 96 Gothenburg, Sweden}

\email{sektnan@chalmers.se}

\address{Carl Tipler, Univ Brest, UMR CNRS 6205, Laboratoire de Mathématiques de Bretagne
Atlantique, France}
\email{Carl.Tipler@univ-brest.fr}

\begin{document}

\begin{abstract} 
We study the zero locus of the Futaki invariant on $K$-polystable Fano threefolds, seen as a map from the K\"ahler cone to the dual of the Lie algebra of the reduced automorphism group. We show that, apart from families $3.9, 3.13, 3.19, 3.20, 4.2, 4.4, 4.7$ and $5.3 $ of the Iskovskikh--Mori--Mukai classification of Fano threefolds, the Futaki invariant of such manifolds vanishes identically on their K\"ahler cone. In all cases, when the Picard rank is greater or equal to two, we exhibit explicit $2$-dimensional differentiable families of K\"ahler classes containing the anti-canonical class and on which the Futaki invariant is identically zero. As a corollary, we deduce the existence of non K\"ahler--Einstein cscK metrics on all such Fano threefolds.
\end{abstract}

\maketitle

\section{Introduction} 
The Futaki invariant was introduced by Akito Futaki (\cite{Futaki83,FutBook}) as an obstruction to the existence of  K\"ahler--Einstein metrics on Fano manifolds. Its definition extends to any compact polarised K\"ahler manifold, and its vanishing is a necessary condition for the existence of a constant scalar curvature K\"ahler metric (cscK for short) in a given K\"ahler class. 

In this note, we study the zero locus of the Futaki invariant, seen as a map from the K\"ahler cone to the dual of the Lie algebra of the reduced automorphism group (see Section \ref{sec:prelim} for the definitions). This locus is fully understood for Fano surfaces from the works \cite{TianYau87,Tian90,WangZhou}, which we recall in Section \ref{sec:fanosurfaces}. Here we will focus on $K$-polystable Fano threefolds. The description of this class of manifolds has seen recently great progress, in particular with \cite{fanothreefolds} (see also references therein). 

Relying on a case by case analysis, our little contribution to Fanography is the following :
\begin{theorem}
 \label{theo:intro}
 Let $(X,-K_X)$  be a $K$-polystable Fano threefold that belongs to family N\textdegree $\cN$, with 
$$\cN\notin\lbrace 3.9, 3.13, 3.19, 3.20, 4.2,  4.4, 4.7, 5.3 \rbrace.$$ 
 Then, the Futaki invariant of $X$ vanishes identically on its K\"ahler cone. 
\end{theorem}
Note that when $\Aut(X)$ is finite or when the Picard rank $\rho(X)=1$, the Futaki invariant vanishes identically on the K\"ahler cone, as soon as $X$ is $K$-polystable in the second case. From the classification in \cite{fanothreefolds}, there exists $33$ families of Fano threefolds with $\rho(X)\geq 2$ that admit members which are $K$-polystable with respect to the anti-canonical polarisation and which have infinite automorphism group. We verify that of these, only $8$ families might have members with classes on which the Futaki invariant does not vanish. Further, for these $8$ families, we provide explicit $2$-dimensional families of K\"ahler classes that contain $c_1(X)$ and on which the Futaki invariant vanishes. 
\begin{theorem}
 \label{theo:intro2}
 Let $(X,-K_X)$  be a $K$-polystable Fano threefold that belongs to family N\textdegree $\cN$, with 
$$\cN\in\lbrace 3.9, 3.13, 3.19, 3.20, 4.2,  4.4, 4.7, 5.3 \rbrace.$$ 
 Then, there is at least a 2-dimensional family of K\"ahler classes on $X$, containing $c_1(X)$, where the Futaki invariant vanishes.
\end{theorem}

From the LeBrun--Simanca openness theorem (\cite{LeSim93}), we deduce the following corollary.
\begin{corollary}
 \label{cor:intro}
 Let $X$ be a $K$-polystable Fano threefold with Picard rank $\rho(X)\geq 2$. Then $X$ admits a $2$-dimensional family of cscK metrics parametrised by a $2$-dimensional family of K\"ahler classes containing $c_1(X)$.
\end{corollary}

\begin{remark}
For $K$-polystable members of the families $\cN\in\lbrace  4.2, 4.4, 4.7 \rbrace$ with infinite automorphism group, we actually show that there is a $3$-dimensional family of K\"ahler classes near $-K_X$ with vanishing Futaki invariant. 
\end{remark}

\begin{remark}
 Our results should be compared with the recent \cite{OSY}, where the Futaki invariant of Bott manifolds is studied. In contrast to our results, which guarantee the vanishing of the Futaki invariant in many cases, it is shown in \cite{OSY} that the only Bott manifolds for which the Futaki invariant vanishes on the whole K\"ahler cone are isomorphic to products of projective lines. The key observation to prove our results is the existence of enough discrete symmetries that preserve every K\"ahler class on Fano threefolds, which in the majority of the cases considered here will be responsible for the vanishing of the Futaki invariant. 
\end{remark}

\subsection*{Acknowledgments}  
The authors would like to thank Hendrik S\"u{\ss} for kindly answering our questions on complexity one Fano threefolds. 
 CT is partially supported by the grants MARGE ANR-21-CE40-0011 and BRIDGES ANR--FAPESP ANR-21-CE40-0017. LMS is funded by a Marie Sk\l{}odowska-Curie Individual Fellowship, funded from the European Union's Horizon 2020 research and innovation programme under grant agreement No 101028041.
 
\subsection*{Notations and conventions} Throughout the paper, for a compact K\"ahler manifold $X$, we will denote by $\Aut(X)$ (respectively $\Aut_0(X)$) its automorphism group  (respectively the connected component of the identity of the reduced automorphism group of $X$), and by $\aut(X)$ the Lie algebra of $\Aut(X)$. If $Z\subset X$ is a subvariety (not necessarily connected), $\Aut(X,Z)$ stands for elements in $\Aut(X)$ that leave $Z$ globally invariant. We denote by $\cK_X$ the K\"ahler cone of $X$. We will identify a divisor $D$ with $\cO(D)$, and use the notation $c_1(D)$ for its first Chern class. 

 \section{Preliminaries}
 \label{sec:prelim}
 Let $X$ be a compact K\"ahler manifold, and $\Omega\in \cK_X$ a K\"ahler class on $X$. We denote the Futaki invariant of $(X,\Omega)$ by 
 $$
 \begin{array}{cccc}
  
 \Fut_{(X,\Omega)} :  & \aut_0(X) &  \to & \C \\
                &    v      &  \mapsto & -\int_X \, f_{v,g}\, \mathrm{scal}_g\, d\mu_g,
 \end{array}
 $$
  where $\aut_0(X)$ is the Lie algebra of the reduced automorphism group of $X$, $g$ denotes a K\"ahler metric with K\"ahler form in $\Omega$ and volume form $d \mu_g$,  $f_{v,g}$ is the normalised holomorphy potential of $v$ with respect to $g$, and $\mathrm{scal}_g$ denotes the scalar curvature of $g$ (see e.g. \cite{CalabiII}, \cite[Section 3.1]{LeSim94} or \cite[Chapter 4]{gauduchon} for this formulation of the Futaki invariant, initially introduced in \cite{Futaki83}). 
  
  We will be interested in $K$-polystable Fano manifolds, or equivalently Fano manifolds admitting a K\"ahler--Einstein metric of positive curvature by the resolution of the Yau--Tian--Donaldson conjecture (\cite{YauOpenPb,Tian97,DonaldsonToric,CDS}). For such manifolds, by Matsushima's result (\cite{matsushima57}), and from Bochner's formula (see \cite[Section 3.6]{gauduchon}), we have $\aut_0(X)=\aut(X)$. We will therefore consider the Futaki invariant as a map 
 $$
 \Fut_X : \cK_X \to \aut(X)^*.
 $$
 By construction, $\Fut_X$ vanishes on any class that admits a cscK metric, and it is then straightforward that $\Fut_X\equiv 0$ whenever $X$ is a $K$-polystable Fano manifold with Picard rank $1$, or when the automorphism group of $X$ is finite.
 
\subsection{The case of smooth Del Pezzo surfaces}
\label{sec:fanosurfaces}
We refer here the reader to \cite[Section 2]{cps19} and \cite[Section 2]{fanothreefolds}.
  If $X$ is a smooth Del Pezzo surface with infinite automorphism group, then $K_X^2\in\lbrace 6, 7 , 8 , 9 \rbrace$. Moreover, it is $K$-polystable and of Picard rank $\rho(X)\geq 2$ if and only if $X=\P^1\times\P^1$ or $K_X^2=6$, i.e. when $X$ is a blow-up of $\P^2$ along three non-collinear points (\cite{TianYau87,Tian90}). In the first case, $X$ admits a product cscK metric in each class, and $\Fut_X\equiv 0$, while in the latter case, the vanishing locus of $\Fut_X$ is described in \cite[Section 5]{WangZhou} (see Section \ref{sec:family53} for the exact description). 

 \subsection{Further properties of the Futaki invariant}
The key property that we will use is the invariance of $\Fut_X$ under the $\Aut(X)$-action. This was already used in \cite[Chapter 3]{FutBook} to show the vanishing of $\Fut_X$ on specific examples. 

We will use the following proposition repeatedly.
 \begin{proposition}
 \label{prop:AdinvarianceFutCar}
  Let $(X,\Omega)$ be a polarised Fano manifold. Assume that there is $\tau\in \Aut(X)$ and $v\in \aut(X)$ such that 
  \begin{itemize}
   \item[(i)] $\tau^*\Omega=\Omega$,
   \item[(ii)] there is $c\in\C^*\setminus \lbrace 1\rbrace$ with  $\Ad_\tau(v)=c\cdot v$.
  \end{itemize}
Then  $\mathrm{Fut}_{(X,\Omega)}(v)=0$.
 \end{proposition}
\begin{proof}
 This follows from the $\Ad$-invariance of the Futaki invariant, which implies that
 $$
 \Fut_{(X,(\tau^{-1})^*\Omega)}(\Ad_\tau(v))=\Fut_{(X,\Omega)}(v),
 $$
 see \cite[Chapter 3]{FutBook} or \cite[Section 3.1]{LeSim94}.
\end{proof}
\begin{remark}
 The anti-canonical class $c_1(X)$ is always $\Aut(X)$-invariant.
\end{remark}

As an application, we have the following useful corollary :
\begin{corollary}
 \label{cor:BlowupandAdinvarianceFutaCar}
 Let $\pi : X \to Y $ be the blow-up of a smooth Fano manifold $Y$ along smooth and disjoint subvarieties $ Z_i\subset Y$. Assume that there is a finite group $G\subset \Aut(Y)$ such that :
 \begin{itemize}
  \item[(i)] Each $Z_i$ is $G$-invariant;
  \item[(ii)] Each class $\Omega\in H^{1,1}(Y,\R)$ is $G$-invariant;
  \item[(iii)] For any $v\in\aut(Y)$ that lifts to $X$, there is $\tau\in G$ and $c\in \C^*\setminus \lbrace 1 \rbrace$ such that $\Ad_\tau(v)=c\cdot v$.
 \end{itemize}
 Then $\Fut_X\equiv 0$.
\end{corollary}
\begin{proof}
From hypothesis (i), the $G$-action on $Y$ lifts to a $G$-action on $X$.  The vector space $H^{1,1}(X,\R)$ is spanned by the pullback of the classes in $H^{1,1}(Y,\R)$ and the exceptional divisors of $\pi$.  By hypothesis (i) and (ii), any class in $H^{1,1}(X,\R)$ is then $G$-invariant. The Lie algebra $\aut(X)$ is spanned by lifts of elements in $\aut(Y)$ that preserve the $Z_i$'s. For any such element, the identity $\Ad_\tau(v)=c\cdot v$ holds on $X\setminus \bigcup_i \pi^{-1}(Z_i)$, hence on $X$, by continuity. The result follows from Proposition \ref{prop:AdinvarianceFutCar}.
\end{proof}
\begin{remark}
In practice, we will mainly use Corollary \ref{cor:BlowupandAdinvarianceFutaCar} with 
$$G\simeq \Z/2\Z,\: \aut(X)\simeq \C,\:  c=-1.$$
To prove item $(i)$ of Proposition \ref{prop:AdinvarianceFutCar} or item $(ii)$ of Corollary \ref{cor:BlowupandAdinvarianceFutaCar}, we will use the fact that in homogeneous coordinates, the Fubini--Study metric 
$$
\om_{FS}=\frac{i}{2}\del\delb \log(\vert z \vert^2),
$$
and hence its class $[\om_{FS}]\in H^{1,1}(\P^n,\R)$, is invariant under the $\mathfrak{S}_{n+1}$-action on $\P^n$ by permutation of the homogeneous coordinates.
\end{remark}
 
\subsection{The list to check}
From the discussion in the beginning of this section, to prove Theorem \ref{theo:intro}, it is enough to consider $K$-polystable Fano threefolds with infinite automorphism group and Picard rank $\rho(X)\geq 2$. From \cite[Section 6]{fanothreefolds}, this reduces to Fano threefolds in family N\textdegree $\cN$, for 
$$
\cN\in\left\{ \begin{array}{lll}
             2.20, 2.21, 2.22,2.24, 2.27,  2.29, 2.32 ,2.34 ,3.5 , 3.8, 3.9, \\   
              3.10, 3.12 ,3.13, 3.15, 3.17 , 3.19, 3.20, 3.25 , 3.27, 4.2, 4.3,\\
              4.4,
 4.6 ,4.7, 4.13 , 5.1 ,5.3, 6.1, 7.1, 8.1, 9.1, 10.1  
              \end{array}
 \right\}.
$$
The strategy of the proof is then direct -- we will use the invariance of $\Fut_X$ to show its vanishing on $\cK_X$ using a case by case study. For $X$ belonging to family N\textdegree$\cN$ with 
$$
\cN\in \left\{ \begin{array}{lll}
             2.20, 2.21, 2.22,2.24, 2.27,  2.29, 2.32 ,2.34 ,  \\   
              3.5 , 3.8,3.10, 3.12 , 3.15, 3.17 ,  3.25 , 3.27,  \\
 4.3,4.6 , 4.13 , 5.1 , 6.1, 7.1, 8.1, 9.1, 10.1  
              \end{array}\right\},
$$
we will see that $\Fut_X\equiv 0$.  

For the remaining $8$ families, we will obtain explicit K\"ahler classes of the form 
$$
c_1(X)+\ep c_1(D)\in\cK_X
$$
with $D\subset X$ a divisor and $\ep\in\R$ a parameter such that $\Fut_{(X,c_1(X)+\ep c_1(D))}=0$. As the vanishing of the Futaki invariant is preserved under scaling of the K\"ahler metric, this provides the $2$-dimensional families of K\"ahler classes with vanishing Futaki invariant alluded to in the introduction. Corollary \ref{cor:intro} then follows from LeBrun--Simanca's openness theorem \cite{LeSim93}, which asserts that the locus in the K\"ahler cone of K\"ahler classes that admit an extremal metric in the sense of Calabi (\cite{Calabi}) is open, together with the characterisation of cscK metrics amongst extremal metrics as the ones with zero Futaki invariant (\cite{CalabiII}).

\section{Families with $\aut(X)\simeq \mathfrak{sl}_n(\C)$}
\label{sec:sln}
Here we will consider families N\textdegree $\cN$, with 
$$\cN\in\lbrace  2.27, 2.32, 3.17,  4.6, 6.1, 7.1, 8.1, 9.1, 10.1 \rbrace.$$
The Lie algebra $\mathfrak{sl}_n(\C)$ is simple, hence equal to its derived ideal $[\mathfrak{sl}_n(\C),\mathfrak{sl}_n(\C)]$. As the Futaki invariant is a character from $\aut(X)$ to $\C$ (see e.g. \cite[Chapter 3]{FutBook} or \cite[Section 3.1]{LeSim94}), it vanishes on the derived ideal $[\aut(X),\aut(X)]$. Hence, if $\aut(X)\simeq\mathfrak{sl}_n(\C)$, $[\aut(X),\aut(X)]=\aut(X)$ and the Futaki invariant vanishes identically on the whole K\"ahler cone of $X$. From 
$$
\mathrm{PGL}_n(\C)\simeq \mathrm{SL}_n(\C)/\mu_n,
$$
the Lie algebra of $\mathrm{PGL}_n(\C)$ is $\mathfrak{sl}_n(\C)$. From \cite[Section 6, Big Table]{fanothreefolds}, this settles the case of all the $K$-polystable Fano threefolds in families N\textdegree $\cN$, with 
$$
\cN\in\lbrace 2.27, 2.32, 3.17, 4.6, 6.1, 7.1, 8.1, 9.1, 10.1\rbrace,
$$
and also some cases in families $\lbrace 2.21, 3.13 \rbrace$.

\section{Products}
\label{sec:products}
Next, we consider families N\textdegree  2.34, N\textdegree 3.27  and N\textdegree5.3, which are products of lower dimensional Fano manifolds.

\subsection{Families 2.34 and 3.27}
The unique members in these two families are $\P^1\times\P^2$ and $\P^1 \times \P^1 \times \P^1$, which both carry a product of cscK metrics in any class, and thus has vanishing Futaki character for any K\"ahler class.

\subsection{Family 5.3}
\label{sec:family53}
The unique Fano threefold in family 5.3 is $\P^1\times S_6$, where $S_6$ is the Del Pezzo surface with $K_{S_6}^2=6$. It is $K$-polystable as a product of K\"ahler--Einstein manifolds from \cite{TianYau87,Tian90}. The surface $S_6$ is the unique (up to isomorphism) toric surface obtained by blowing-up $\P^2$ in the three fixed points under the torus action. We denote by $H$ (the strict transform of) a generic hyperplane and $D_1, D_2$ and $D_3$ the three exceptional divisors in $S_6$. From \cite[Section 5, Proposition 5.2 and Remark 5.1.(iii)]{WangZhou}, the Futaki invariant of $S_6$ vanishes exactly in the following families of K\"ahler classes
$$
3c_1(H)-ac_1(D_1)-bc_1(D_2)-(3-a-b)c_1(D_3)
$$
and
$$
3c_1(H)-c(c_1(D_1)+c_1(D_2)+c_1(D_3)),
$$
where $a,b,c$ are positive constants satisfying $a+b<3$ and $c<\frac{3}{2}$. As the Futaki invariant vanishes on $\P^1$, we easily deduce the vanishing locus of the Futaki invariant on $X=\P^1\times S_6$. In particular, as $c_1(X)=c_1(\P^1)+c_1(S_6)$, and as $c_1(S_6)=3c_1(H)-(c_1(D_1)+c_1(D_2)+c_1(D_3))$, we deduce the existence of differentiable families of K\"ahler classes on $X$ containing $c_1(X)$ for which the Futaki invariant vanishes identically.

\section{Blow-ups of projective space}
In this section we address families N\textdegree $\cN$, with $$\cN\in\lbrace 2.22, 3.12, 3.25\rbrace.$$ All the members of these families are obtained by blowing up certain curves in projective space $\P^3$.

\subsection{Family 2.22}
Members of the family 2.22 of Fano threefolds are obtained as blowups of certain curves in $\P^3$. More precisely, let $\Phi : \P^1 \times \P^1 \to \P^3$ be the Segre embedding
$$
([x:y],[u:v]) \mapsto [xu : xv : yu : yv].
$$
The image of $\Phi$ is the surface 
$$
S=\lbrace z_0 z_3 - z_1 z_2=0 \rbrace.
$$
A Fano threefold $X$ is in the family 2.22 if it is the blowup of the image via $\Phi$ of a curve $\check C$ with $\mathcal{O}(\check C) = \mathcal{O}(3,1).$ Such $X$ have Picard rank $2$, generated by the line bundle associated to the proper transform of a hyperplane and of that generated by the exceptional divisor $E$ of the blowup. The $K$-polystability of members of this family (with respect to the anticanonical polarisation) is discussed in detail in \cite{cheltsovpark22}. 

Up to biholomorphism, there is a unique member $X_0$ of this family with infinite automorphism group.
It is $K$-polystable, and can be obtained by picking the curve $\check C$ to be $\check C_0 = \lbrace ux^3-vy^3=0\rbrace$, so that 
$$
X_0 = \Bl_{C_0} \P^3,
$$
where $C_0 = \Phi(\check C_0).$ The $\C^*$-action 
$$\lambda\cdot([z_0 : z_1 : z_2 : z_3]) = [\lambda z_0 : \lambda^4 z_1 : z_2 : \lambda^3 z_3]$$
preserves $C_0$ and so lifts to $X_0$. This generates $\Aut_0 (X_0)$ (see \cite[Lemma 6.13]{cps19}). 

The curve $C_0$ is a rational curve, which can e.g. be seen by applying the Riemann--Hurwitz formula to the restriction to $\check C_0 \subset \P^1 \times \P^1$ of the projection to the second factor. An explicit parametrisation is given by
$$
[\tau_0 : \tau_1] \mapsto [\tau_0\tau_1^3 : \tau_0^4 : \tau_1^4 : \tau_1 \tau_0^3].
$$
Note that the action of the involution $\tau$ given by
$$
\tau\cdot([z_0 : z_1 : z_2 : z_3]) = [z_3 : z_2 : z_1 : z_0]
$$
on $\P^3$ preserves $C_0$ and so lifts to $X_0$. We can then apply Corollary \ref{cor:BlowupandAdinvarianceFutaCar} to the blow-up $X_0\to \P^3$ with the group $G=\langle \tau \rangle \simeq \Z/2\Z$, which implies the vanishing of the Futaki invariant on the K\"ahler cone of $X_0$.

\subsection{Family 3.12}
From \cite[Section 5.18]{fanothreefolds}, the only element in Family 3.12 with infinite automorphism group is given, up to isomorphism, by $X=\mathrm{Bl}_{L\cup C}(\P^3)$ the blow up of $\P^3$ along the disjoint curves 
$$
L=\lbrace x_0=x_3=0 \rbrace \subset \P^3
$$
and
$$
C=\lbrace [s^3 : s^2 t : st^2 : t^3 ], [s : t]\in \P^1 \rbrace\subset \P^3.
$$
The reduced automorphism group of $X$ is isomorphic to $\C^*$, and its action is given by the lift of the $\C^*$-action on $\P^3$ described by
$$
\lambda\cdot ([x_0:x_1:x_2:x_3] ) =[x_0 : \lambda x_1 : \lambda^2 x_2 : \lambda^3 x_3] .
$$
Then, we can consider the $\Z/2\Z$-action given by 
$$
\tau([x_0:x_1 : x_2 : x_3]) = [x_3:x_2:x_1:x_0].
$$
The group generated by $\tau$ in $\Aut(\P^3)$ satisfies hypothesis $(i)-(iii)$ from Corollary \ref{cor:BlowupandAdinvarianceFutaCar}, and we deduce that the Futaki invariant of $X$ vanishes on the whole K\"ahler cone.

\subsection{Family 3.25}
The Fano threefold $X$ in family 3.25 is the blow-up of $\P^3$ in two disjoint lines. It is $K$-polystable from \cite{BatSel,WangZhu}. We can assume the two blown-up lines are $\lbrace x_1=x_2=0\rbrace\subset \P^3$ and $\lbrace x_3=x_4=0 \rbrace\subset \P^3$. One has 
$$
\Aut_0(X)\simeq \PGL_{(2,2)}(\C)\simeq \GL_2(\C)\times\GL_2(\C)/\C^*,
$$
where the first (resp. second) $\GL_2(\C)$ factor acts linearly on the coordinates $(x_1,x_2)$ (resp. on $(x_3,x_4)$) while the $\C^*$-action corresponds to homotheties on $\C^4$ (see \cite[Section 4]{cps19}). The Lie algebra $\aut(X)$ of $\Aut_0(X)$ fits in an exact sequence
$$
0 \to \C \to  \gl_2(\C)\oplus\gl_2(\C) \to \aut(X) \to 0.
$$
We also have the sequence induced by the trace map $\gl_2(\C) \to \C$ :
$$
0\to \sl_2(\C) \to \gl_2(\C) \to \C \to 0,
$$
from which we deduce the sequence of vector spaces 
$$
0 \to \C \to  (\sl_2(\C)\oplus\C)\oplus(\sl_2(\C)\oplus\C) \to \aut(X) \to 0.
$$
From the discussion in Section \ref{sec:sln}, the Futaki invariant of $X$ will vanish on the $\sl_2(\C)$-factors that project to $\aut(X)$. Hence, it is enough to test the vanishing of the Futaki invariant on the generators of the remaining two $\C^*$-actions modulo homotheties, which are induced by:
$$
(\lambda,\mu)\cdot ([x_0:x_1:x_2:x_3])=[\lambda x_0: x_1 : \mu x_2 : x_3],
$$
where  $(\lambda,\mu)\in (\C^*)^2 $. We can consider the finite group $G$ generated by the reflections 
$$
\tau( [x_0:x_1:x_2:x_3])= [x_1:x_0:x_2:x_3]
$$
and 
$$
\sigma( [x_0:x_1:x_2:x_3])= [x_0:x_1:x_3:x_2].
$$
This group preserves the two blown-up lines, while the adjoint action of $\tau$ (resp. $\sigma$) sends the generator of the $\lambda$-action (resp. the $\mu$-action) to its inverse. Hence, from Corollary \ref{cor:BlowupandAdinvarianceFutaCar}, we see that the Futaki invariant of $X$ vanishes on its whole K\"ahler cone.

\section{Blow-ups of products of projective spaces}
In this section, we will consider families N\textdegree $\cN$, with $$\cN\in\lbrace 3.5, 4.3, 4.13 \rbrace.$$ These are obtained as blowups of products of projective spaces.

\subsection{Family 3.5}
From \cite[Section 5.14]{fanothreefolds}, the only element in Family 3.5 with infinite automorphism group is given, up to isomorphism, by $X=\mathrm{Bl}_{C}(\P^1\times \P^2)$ the blow up of $\P^1\times\P^2$ along the curve $C=\psi(\check C)$ given by the image of
$$
\check C=\lbrace ux^5+vy^5=0 \rbrace \subset \P^1\times\P^1
$$
via the map 
$$
\begin{array}{cccc}
 \psi : & \P^1\times\P^1 & \to & \P^1\times \P^2\\
        &    ([u:v],[x:y])   & \mapsto & ([u:v],[x^2:xy:y^2]).  
\end{array}
$$
Then, $\Aut_0(X)\simeq\C^*$, where the $\C^*$-action is generated by the lift to $X$ of the action
$$
\lambda\cdot([u:v],[x_0:x_1:x_2])=([\lambda^5 u : v],[x_0:\lambda x_1:\lambda^2 x_2]).
$$
We also have a $\Z/2\Z$-action induced by
$$
\tau([u:v],[x_0:x_1:x_2])=([v:u],[x_2:x_1:x_0]).
$$
Those actions come respectively from the actions
$$
\lambda\cdot([u:v],[x:y])=([\lambda^5 u : v],[x:\lambda y])
$$
and 
$$
\tau([u:v],[x:y])=([v:u],[y:x])
$$
on $\P^1\times\P^1$, with respect to which $\psi$ is equivariant. Then, we see that $C$ is $\tau$-invariant, as well as the classes $\pi_i^*[\om_{FS}^i]$, where $\pi_1 : \P^1\times\P^2 \to \P^1$ and $\pi_2 : \P^1\times \P^2 \to \P^2$ denote the projections and $\om_{FS}^i$ stands for the Fubini--Study metric on $\P^i$. Finally, identifying $\lambda\in\C^*$ with its action, we have $\tau\circ\lambda\circ \tau^{-1}=\lambda^{-1}$. Hence, hypothesis $(i)-(iii)$ from Corollary \ref{cor:BlowupandAdinvarianceFutaCar} are satisfied, and the Futaki invariant of $X$ vanishes for any K\"ahler class.

\subsection{Family 4.3}
Following \cite[Section 5.21]{fanothreefolds}, up to isomorphism, the unique Fano threefold in Family 4.3 is the blow-up of $\P^1\times \P^1\times \P^1$ along
$$
C=\lbrace x_0y_1 - x_1 y_0= x_0z_1^2+x_1z_0^2=0 \rbrace 
$$
where $[x_0:x_1]$, $[y_0:y_1]$ and $[z_0:z_1]$ denote the homogeneous coordinates on the first, second and last factor respectively. We have $\Aut_0(X)\simeq \C^*$ where the action is given by the lift of the $\C^*$-action on $\P^1\times \P^1\times \P^1$ given by
$$
\lambda\cdot([x_0:x_1],[y_0:y_1],[z_0: z_1])=([x_0:\lambda^2x_1],[y_0:\lambda^2y_1],[z_0:\lambda z_1]).
$$
The involution
$$
\tau ([x_0:x_1],[y_0:y_1],[z_0: z_1])=([x_1:x_0],[y_1:y_0],[z_1: z_0])
$$
preserves $C$ and the $(1,1)$-classes on $C$ given by $\iota^*_j[\om_{FS}]$, for $\iota_j$ the composition of the inclusion $C\subset \P^1\times\P^1\times\P^1$ and the projection on the $j$-th factor. The adjoint action of $\tau$ maps the generator of the $\C^*$-action to its inverse, so Proposition \ref{prop:AdinvarianceFutCar} applies and the Futaki invariant of $X$ vanishes identically.

\subsection{Family 4.13}
From \cite[Section 5.22]{fanothreefolds}, the only element in Family 4.13 with infinite automorphism group is given, up to isomorphism, by $X=\mathrm{Bl}_C(\P^1\times \P^1\times \P^1)$ the blow up of $\P^1\times \P^1\times \P^1$ along the curve 
$$
C=\lbrace x_0y_1-x_1y_0=x_0^3z_0+x_1^3z_1=0\rbrace.
$$
The reduced automorphism group of $X$ is isomorphic to $\C^*$, and its action is given by the lift of the $\C^*$-action on $\P^1\times \P^1\times \P^1$ described by
$$
\lambda\cdot ([x_0:x_1],[y_0:y_1],[z_0:z_1] ) =([\lambda x_0:x_1],[\lambda y_0:y_1],[\lambda^{-3} z_0:z_1] ).
$$
Then, we can consider the $\Z/2\Z$-action given by 
$$
\tau([x_0:x_1],[y_0:y_1],[z_0:z_1]) = ([x_1:x_0],[y_1:y_0],[z_1:z_0]).
$$
Clearly, this action satisfies hypothesis $(i)-(iii)$ from Corollary \ref{cor:BlowupandAdinvarianceFutaCar} (notice that $\tau\circ \lambda \circ \tau^{-1}=\lambda^{-1}$, identifying $\lambda$ with the induced action), from which we deduce the vanishing of the Futaki invariant of $X$ for any K\"ahler class.

\section{Blow-ups of a smooth quadric}
In this section, we consider families N\textdegree $\cN$, with $$\cN\in\lbrace 2.21, 2.29, 3.10, 3.15, 3.19, 3.20, 4.4, 5.1  \rbrace.$$

\subsection{Family 2.21}
This family is somewhat similar to the Mukai--Umemura family 1.10. In addition to members of the family with discrete automorphism group, there is a one-dimensional family with automorphism group containing a semi-direct product of $\C^*$ and $\Z/2\Z$, one member which admits an effective $\textnormal{PGL}_2$-action and one member which has a reduced automorphism group $\mathbb{G}_a$. The first two of these are $K$-polystable for the anti-canonical polarisation, whereas the last does not have a reductive automorphism group and is therefore not $K$-polystable.

The members that admit an effective $\mathbb{G}_m$-action can be described as follows (see \cite[Section 5.9]{fanothreefolds}). Let $C$ be the quartic rational curve in $\P^4$ given as the image of the map $\P^1 \to \P^4$ given by
$$
[p:q] \mapsto [p^4: p^3 q : p^2q^2 : pq^3 : q^4].
$$
For $t \notin \{0, \pm 1\}$, let $Q_t$ be the smooth hypersurface 
$$
Q_t = V(z_1z_3 -t^2z_0z_4 + (t^2-1)z_2^2).
$$
Note that $C \subset Q_t$ for any $t$. Let $X_t = \Bl_C( Q_t)$. Then $X_t$ is one of the members that admit an effective $\C^*$-action (including the member with an effective $\textnormal{PGL}_2$-action, which corresponds to $t=\pm \frac{1}{2}$). Note that $X_t$ has Picard rank $2$, generated by a hyperplane $H$ and the exceptional divisor $E$ of the blowup. 

The $\C^*$-action given by 
$$
\lambda\cdot([z_0:z_1:z_2:z_3:z_4]) = [z_0: \lambda z_1: \lambda^2 z_2: \lambda^3 z_3: \lambda^4 z_4]
$$
preserves $C$ and $Q_t$, as does the involution
$$
\tau([z_0:z_1 : z_2 : z_3 : z_4] )= [z_4 : z_3 : z_2 : z_1 : z_0].
$$ 
The lifts of these generate the effective actions of $\C^*\rtimes\Z/2\Z$ on $X_t$. As $\tau$ preserves $C$, the class $[\om_{FS}]_{\vert X_t}$, and sends a generator of the $\C^*$-action to its inverse by conjugation, Proposition \ref{prop:AdinvarianceFutCar} shows that the Futaki invariant of $X_t$ vanishes on its whole K\"ahler cone (note that the case $t=\pm \frac{1}{2}$, with $\aut(X_t)\simeq \sl_2(\C)$, was dealt with in Section \ref{sec:sln}).

\subsection{Family 2.29}
There is a unique smooth  Fano threefold $X$ in family 2.29. It is isomorphic to the blow-up of
$$
Q=\lbrace x_0^2+x_1x_2 +x_3x_4=0 \rbrace \subset \P^4.
$$
along the smooth conic
$$
C=\lbrace x_0^2+x_1x_2=x_3=x_4=0 \rbrace \subset Q.
$$
It is $K$-polystable (see \cite{SussKE,SussPictureBook,IltenSuss}) and the group $\Aut_0(X)$ is isomorphic to $\C^*\times \PGL_2(\C)$ (see \cite[Lemma 5.8]{cps19}). We then have $\aut(X)\simeq \C\oplus \sl_2(\C)$. From the discussion in Section \ref{sec:sln}, the Futaki invariant of $(X, [\om])$ vanishes on the $\sl_2(\C)$-component of $\aut(X)$
for any K\"ahler class $[\om]$. Thus, to check the vanishing of the Futaki invariant, it remains to check the vanishing on the $\C$-component of $\aut(X)$. From \cite[Lemma 5.7]{cps19}, the $\C^*$-component of $\Aut_0(X)$ can be identified with the pointwise stabiliser of $C$ in $\Aut_0(Q)$. This is then the $\C^*$-action induced by
$$
\lambda\cdot ([x_0:x_1:x_2:x_3:x_4])=[x_0:x_1:x_2:\lambda x_3:\lambda^{-1}x_4].
$$
We then introduce the involution 
$$
\tau([x_0:x_1:x_2:x_3:x_4])=[x_0:x_2:x_1:x_4:x_3].
$$
This automorphism of $\P^4$ preserves $Q$ and $C$ and lifts to an automorphism of $X$.  
Its adjoint action maps a generator of the $\C^*$-action of interest to its inverse, and by Corollary \ref{cor:BlowupandAdinvarianceFutaCar}, we deduce the vanishing of the Futaki invariant of $X$ for any K\"ahler class.

\subsection{Family 3.10}
Let $X$ be a $K$-polystable element in the family 3.10 such that $\Aut(X)$ is infinite. Then, from \cite[Section 5.17]{fanothreefolds}, up to isomorphism, we may assume that $X=\Bl_{C_1\cup C_2}(Q_a)$ is the blow-up of the quadric 
$$
Q_a=\lbrace w^2+xy+zt+a(xt+yz)=0 \rbrace \subset \P^4
$$
along the two disjoint smooth irreducible conics $C_1\subset Q_a$ and $C_2\subset Q_a$ given by 
$$
C_1=\lbrace w^2+zt=x=y=0 \rbrace
$$
and
$$
C_2=\lbrace w^2+xy=z=t=0 \rbrace
$$
where $[x,y,z,t,w]$ stand for the homogeneous coordinates on $\P^4$ and where $a\in\C\setminus \lbrace -1, +1 \rbrace$ is a complex parameter. Moreover, for $a=0$, $\Aut_0(X)\simeq (\C^*)^2$ and for $a\neq 0$, $\Aut_0(X)\simeq \C^*$.
\subsubsection{Case $a=0$}
In this situation, the $(\C^*)^2$-action on $X$ is the lift of the action on $Q_0$ induced by the following formula, for $(\alpha,\beta)\in(\C^*)^2$ :
$$
(\alpha,\beta)\cdot ([x:y:z:t:w])=[\alpha x : \alpha^{-1}y : \beta z : \beta^{-1} t : w].
$$
Consider the group $G=\Z/2\Z \times \Z/2\Z$ generated by $(\sigma,\tau)$ defined by
$$
\sigma([x:y:z:t:w])=[y:x:z:t:w]
$$
and 
$$
\tau([x:y:z:t:w])=[x:y:t:z:w].
$$
Then $G \subset \Aut(Q_0)$, and $G$ preserves $C_1$ and $C_2$. It also leaves invariant the class $\iota^*[\om_{FS}]$ on $Q_0$, where $\iota : Q_0 \to \P^4$ denotes the inclusion and $\om_{FS}$ the Fubini--Study metric. Hence, hypothesis $(i)$ and $(ii)$ of Corollary \ref{cor:BlowupandAdinvarianceFutaCar} are satisfied. Finally, $\Ad_\sigma(v_1)=-v_1$ and $\Ad_\tau(v_2)=-v_2$, where $v_1$ generates the $\C^*$-action $\alpha \mapsto [\alpha x : \alpha^{-1}y :  z :  t : w]$ while $v_2$ generates the $\C^*$-action $\beta \mapsto [x : y : \beta z : \beta^{-1} t : w]$. Then, Corollary \ref{cor:BlowupandAdinvarianceFutaCar} implies the vanishing of the Futaki invariant on $X$ for any class.

\subsubsection{Case $a\neq 0$}
The same argument as in the previous case applies, where this time the $\C^*$-action of $\Aut_0(X)$ is induced by the diagonal of the above, given by
$$
\alpha\cdot ([x:y:z:t:w])=([\alpha x : \alpha^{-1}y : \alpha z : \alpha^{-1} t : w]).
$$
and the group $G\simeq\Z/2\Z$ is generated by 
$$
\varsigma([x:y:z:t:w])=([y:x:t:z:w]).
$$

\subsection{Family 3.15}
From \cite[Section 5.20]{fanothreefolds}, the only smooth $K$-polystable Fano threefold in family 3.15 is given by the blow-up $X=\Bl_{L\cup C}(Q)\to Q$ of the quadric 
$$
Q=\lbrace x_0^2+2x_1x_2+2x_1x_4+2x_2x_3 =0\rbrace \subset \P^4
$$
along the line 
$$
L=\lbrace x_0=x_1=x_2=0 \rbrace
$$
and the smooth conic (disjoint from $L$)
$$
C=\lbrace x_0^2+2x_1x_2=x_3=x_4=0 \rbrace.
$$
The automorphism group of $X$ satisfies $\Aut_0(X)\simeq \C^*$ with $\C^*$-action given, for $\lambda\in\C^*$, by (the lift of)
$$
\lambda\cdot([x_0:x_1:x_2:x_3:x_4])=[\lambda x_0 : \lambda^2 x_1 : x_2 : \lambda^2 x_3 : x_4 ].
$$
The involution 
$$
\tau([x_0:x_1:x_2:x_3:x_4])=[x_0:x_2:x_1:x_4:x_3]
$$
preserves $Q$, $L$ and $C$. It also leaves the class $\iota^*[\om_{FS}]$ invariant, where $\iota : Q \to \P^4$ is the inclusion. Then, Corollary \ref{cor:BlowupandAdinvarianceFutaCar} applies to $X\to Q$ and $G=\langle \tau \rangle\simeq \Z/2\Z$, so that the Futaki invariant of $X$ identically vanishes on $\cK_X$.

\subsection{Families 3.19 and 3.20}

Consider the smooth quadric Fano threefold 
$$
Q=\lbrace x_0^2+x_1x_2 +x_3x_4=0 \rbrace \subset \P^4.
$$
The family 3.19 (resp. 3.20) is obtained by blowing-up $Q$ in two points (respectively two disjoint lines). More precisely, we can obtain the unique Fano threefold in family 3.19 by considering $X_1$ to be the blow-up of $Q$ along the points
$$
P_1=[0:0:0:1:0]
$$
and 
$$
P_2=[0:0:0:0:1].
$$
The unique  Fano threefold $X_2$ in family 3.20 is the blow-up of $Q$ along the two disjoint lines
$$
L_1=\lbrace x_0=x_1=x_3=0 \rbrace
$$
and 
$$
L_2=\lbrace x_0=x_2=x_4=0 \rbrace.
$$
In both cases, the Fano threefold $X_i$ is $K$-polystable (see \cite{SussKE,SussPictureBook,IltenSuss}) and the group $\Aut_0(X_i)$ is isomorphic to $\C^*\times \PGL_2(\C)$ (see \cite[Section 5]{cps19}). We then have $\aut(X_i) = \C\oplus \sl_2(\C)$. From the discussion in Section \ref{sec:sln}, the Futaki invariant of $(X_i, [\om_i])$ vanishes on the $\sl_2(\C)$-component of $\aut(X_i)$
for any K\"ahler class $[\om_i]$. Therefore, to check the vanishing of the Futaki invariant on $(X_i,[\om_i])$, it remains to check the vanishing on the $\C$-component of $\aut(X_i)$. 

To this aim we introduce the involution 
$$
\tau([x_0:x_1:x_2:x_3:x_4])=[x_0:x_2:x_1:x_4:x_3].
$$
This automorphism of $\P^4$ preserves $Q$, and swaps the two connected components of the blown-up locus in both cases. Therefore, $\tau$ lifts to an automorphism of $X_i$, still denoted $\tau$, for $i\in\lbrace 1,2\rbrace$. Note that on $X_i$, any K\"ahler class of the form 
$$[\om_\ep]:=c_1(X_i)+\ep (c_1(\cO(E_1^i))+c_1(\cO(E_2^i)))$$
is $\tau$-invariant, where $\ep\in \R$ is chosen so that
$$
c_1(X_i)+\ep (c_1(\cO(E_1^i))+c_1(\cO(E_2^i)))>0
$$
and the $E_j^i$'s denote the exceptional divisors of the blow-up $X_i \to Q$.

Next, we investigate how this action interacts with the generator of the $\C^*$-component in $\Aut_0 (X_i)$, to verify that we can apply Proposition \ref{prop:AdinvarianceFutCar} to deduce the vanishing of the Futaki invariant. We do this for the two families separately. 

\subsubsection{Family 3.19}
\label{sec:family319}
We follow the discussion in \cite[Lemma 5.13]{cps19}. An automorphism of $X_1$ comes from an automorphism of $\P^4$ that leaves $Q$ and $\lbrace P_1\rbrace \cup \lbrace P_2 \rbrace$ invariant. By linearity, such an automorphism preserves the line spanned by the two points, and thus its orthogonal complement $\Pi=\lbrace x_3=x_4=0 \rbrace$. It then leaves the conic $C=Q\cap \Pi$ invariant. From \cite[Lemma 5.7]{cps19}, the $\C^*$-component of $\Aut_0(X_1)$ can be identified with the pointwise stabiliser of $C$ in $\Aut_0(Q)$. This is then the $\C^*$-action given by
$$
\lambda\cdot ([x_0:x_1:x_2:x_3:x_4])=[x_0:x_1:x_2:\lambda x_3:\lambda^{-1}x_4].
$$
The adjoint action of $\tau$ maps a generator of this action to its inverse, and by Proposition \ref{prop:AdinvarianceFutCar}, we deduce the vanishing of the Futaki invariant of $(X_1,[\om_\ep])$.

\subsubsection{Family 3.20}
Following the discussion in \cite[Lemma 5.14]{cps19}, the $\C^*$-component of $\Aut_0(X_2)$ is obtained as follows. An element in $\Aut_0(Q,L_1\cup L_2)$ must preserve the linear span of $L_1$ and $L_2$, that is $Q\cap \lbrace x_0=0 \rbrace$. It then leaves invariant 
$$
Q'=\lbrace x_0 = x_1 x_2 + x_3 x_4 =0 \rbrace.
$$
The group $\Aut_0(Q, L_1\cup L_2)$ acts on the family of lines $(\ell_t)_{t\in\P^1}$ in $Q'$ given by 
$$[x_1:x_3] \mapsto [0 : x_1 : t x_3 : x_3 : -tx_1] \subset Q'.$$
The $\C^*$-component of $\Aut_0(X)$ then corresponds to the stabiliser of the lines $L_1=\ell_\infty$ and $L_2=\ell_0$ under this action. In coordinates, the action is given by
$$
\lambda\cdot ([x_0:x_1:x_2:x_3:x_4 ])=[\lambda x_0:x_1:\lambda^2 x_2: x_3:\lambda^2 x_4 ].
$$
As with family 3.19, using the $\tau$-action and the $\mathrm{Ad}$-invariance of the Futaki invariant, we can conclude that the Futaki invariant of $(X_2,[\om_\ep])$ vanishes.

\subsection{Family 4.4}
Up to isomorphism, there is a unique smooth Fano threefold $X$ in family 4.4. Its automorphism group satisfies $\Aut_0(X)\simeq (\C^*)^2$, and it is $K$-polystable from \cite{SussKE,SussPictureBook,IltenSuss}. Recall that the smooth Fano threefold $X_1$ in family 3.19 can be obtained as a blow-up along two points of a smooth quadric $Q\subset \P^4$. We can then realise the manifold $X$ as the blow-up of $X_1$ along the proper transform of the conic that passes through the blown-up points in $Q$. Coming back to our parametrisation in Section \ref{sec:family319}, we can take $Q\subset \P^4$ to be 
$$
Q=\lbrace x_0^2+x_1x_2 +x_3x_4=0 \rbrace \subset \P^4
$$
and the blown-up points to be
$$
P_1=[0:0:0:1:0]
$$
and 
$$
P_2=[0:0:0:0:1].
$$
Then, the conic in $Q$ joining $P_1$ and $P_2$ is 
$$
\cC_1=\lbrace x_1=x_2=x_0^2+x_3x_4 =0 \rbrace \subset Q.
$$
The $(\C^*)^2$-action on $Q$ that lifts to $X$ through the two blow-up maps $X\to X_1 \to Q$ is given in coordinates by 
$$
(\lambda,\mu)\cdot([x_0:x_1:x_2:x_3:x_4])=[x_0:\lambda x_1:\lambda^{-1}x_2:\mu x_3:\mu^{-1}x_4].
$$
Again, the involution
$$
\tau([x_0:x_1:x_2:x_3:x_4])=[x_0:x_2:x_1:x_4:x_3]
$$
preserves $C_1$ and swaps the blown-up points. Arguing as before, we see that the Futaki invariant of $X$ will vanish in classes of the form
$$
c_1(X)+\ep\, c_1(\cE) + \delta\,(c_1(E_1)+c_1(E_2))
$$
for $(\ep,\delta)\in\R^2$ small enough and where $\cE$ is the exceptional divisor of $X\to X_1$, while $E_1$ and $E_2$ are the strict transforms of the exceptional divisors of $X_1\to Q$. Note that after scaling, this gives a $3$-dimensional family in the K\"ahler cone of $X$. 

\subsection{Family 5.1}
From \cite[Section 5.23]{fanothreefolds}, the unique smooth Fano threefold $X$ in family 5.1 is $K$-polystable. It can be described as follows. Consider first the smooth quadric in $\P^4$
$$
Q=\lbrace x_1x_2+x_2x_3+x_3x_1+x_4x_5 = 0 \rbrace \subset \P^4
$$
where we denote by $[x_1:x_2:x_3:x_4:x_5]$ the homogeneous coordinates on $\P^4$. We then fix a smooth conic $C=Q\cap \lbrace x_4=x_5=0 \rbrace \subset Q$ and points $P_1=[1:0:0:0:0]$, $P_2=[0:1:0:0:0]$ and $P_3=[0:0:1:0:0]$ in $Q$. Let $Y \to Q$ the blow-up of $Q$ in the three points $(P_i)_{1\leq i\leq 3}$ and $\check C$ the strict transform of $C$ in $Y$. Then, $X$ is obtained as the blow-up of $Y$ along $\check C$. Its automorphism group  satisfies $\Aut_0(X)\simeq \C^*$, where the $\C^*$-action is the lift of the action defined on $Q$ by
$$
\lambda \cdot ([x_1:x_2:x_3:x_4:x_5]) =[\lambda x_1:\lambda x_2:\lambda x_3:\lambda^2 x_4:x_5].
$$
The manifold $X$ also admits an involution which is the lift of the involution $\tau$ defined on $Q$ by 
$$
\tau ([x_1:x_2:x_3:x_4:x_5]) = [x_1:x_2:x_3:x_5:x_4].
$$
We observe that $\tau$ preserves the K\"ahler class associated to the hyperplane section $H\cap Q$ and fixes $C$, as well as the points $P_1$, $P_2$ and $P_3$. Hence,  all the $(1,1)$-classes on $X$ are invariant under the (lifted) involution. As the adjoint action of $\tau$ maps the generator of the $\C^*$-action to its inverse, we conclude as in \ref{cor:BlowupandAdinvarianceFutaCar} the vanishing of the Futaki invariant of $X$ for all its K\"ahler classes.

\section{Hypersurfaces in $\P^2\times\P^2$ and their blow-ups}
In this section, we consider families N\textdegree $\cN$, with $$\cN\in\lbrace 2.24, 3.8, 4.7  \rbrace.$$
\subsection{Family 2.24}
From \cite{SussKE,SussPictureBook,IltenSuss} (see also \cite[Section 4.7]{fanothreefolds}), the only $K$-polystable element in Family 2.24 with infinite automorphism group is given, up to isomorphism, by
$$
X=\lbrace xu^2+yv^2+zw^2 \rbrace \subset \P^2\times\P^2.
$$
It has  $\Aut_0(X)\simeq (\C^*)^2$, where the action of $(\alpha,\beta)\in (\C^*)^2$ is given by
$$
(\alpha,\beta)\cdot([x:y:z],[u:v:w])=([\alpha^2 x:\beta^2 y:z],[\alpha^{-1}u:\beta^{-1}v:w]).
$$
The group $G = \Z/2\Z \times \Z/2\Z$  acts on $\P^2 \times \P^2$, with the action of $(\sigma,\tau) \in G$  generated by 
$$
\sigma([x:y:z],[u:v:w])=([z:y:x],[w:v:u])
$$
and 
$$
\tau([x:y:z],[u:v:w])=([x:z:y],[u:w:v]).
$$
Note that $G\subset \Aut(X)$ and that the inclusion $\iota : X \to \P^2\times\P^2$ is $G$-equivariant. Hence we deduce that $\iota^*[\om_{FS}^i]$ is $G$-invariant, where $\om_{FS}^i$ denote the Fubini--Study metric on the $i$-th factor. Then, any K\"ahler class on $X$ is $G$-invariant. Denote by $v_1$ (resp. $v_2$) the generator of the $\C^*$-action $\alpha\cdot([x:y:z],[u:v:w])=([\alpha^2 x:y:z],[\alpha^{-1}u:v:w])$ (resp. $\beta\cdot([x:y:z],[u:v:w])=([x:\beta^2 y:z],[u:\beta^{-1}v:w])$) on $X$. A direct computation shows 
$$
\left\{ \begin{array}{ccc}
        \Ad_\sigma(v_1) & = & -(v_1+v_2) \\
        \Ad_\tau(v_2) & = & -(v_1+v_2).
       \end{array}
       \right.
$$
Using $\Ad$-invariance of the Futaki invariant, as discussed in Proposition \ref{prop:AdinvarianceFutCar}, we deduce that for any K\"ahler class $\Omega$ on $X$ :
$$
\left\{ \begin{array}{ccc}
        \Fut_{(X,\Omega)}(v_1) & = & -\Fut_{(X,\Omega)}(v_1)-\Fut_{(X,\Omega)}(v_2) \\
         \Fut_{(X,\Omega)}(v_2) & = & -\Fut_{(X,\Omega)}(v_1)-\Fut_{(X,\Omega)}(v_2),
       \end{array}
       \right.
$$
hence $\Fut_X$ is identically zero.

\subsection{Family 3.8}
From \cite[Section 5.16]{fanothreefolds}, the only element in Family 3.8 with infinite automorphism group is given, up to isomorphism, by $X=\mathrm{Bl}_{C}(Y)$ the blow up of $Y$ along the curve $C$, where
$$
Y=\lbrace (vw + u^2 )x + v^2 y + w^2 z = 0 \rbrace \subset \P^2\times \P^2
$$
is a smooth divisor of degree $(1,2)$ and where $C=\pi_1^{-1}([1:0:0])$, with $\pi_1$ the projection onto the first factor of $\P^2\times \P^2$. The variety $Y$ is the only element in Family 2.24 with infinite automorphism group, and 
$$
\Aut(X)\simeq\Aut(Y)\simeq \C^*\rtimes\Z/2\Z.
$$
More explicitly, the $\C^*$-action is for $\lambda\in\C^*$ given by
$$
\lambda\cdot([x:y:z],[u:v:w])=([x:\lambda^{-2}y:\lambda^2z],[\lambda u: \lambda^2 v : w]),
$$
while the $\Z/2\Z$-action is generated by $\tau$:
$$
\tau([x:y:z],[u:v:w])=([x:z:y],[u:w:v]).
$$
Identifying $\lambda$ with the corresponding element in $\Aut(Y)$, we have $\tau\circ\lambda\circ\tau^{-1}=\lambda^{-1}$, so that item $(iii)$ in Corollary \ref{cor:BlowupandAdinvarianceFutaCar} is satisfied. The inclusion $\iota:Y\to \P^2\times\P^2$ is $\tau$-equivariant, and then the classes $\iota^*[\om_{FS}^i]$ are $\tau$-invariant, for $\om_{FS}^i$ the Fubini--Study metric on each factor of $\P^2\times \P^2$. This shows that hypothesis $(ii)$ from Corollary \ref{cor:BlowupandAdinvarianceFutaCar} holds as well. Finally, the curve $C$ is $\tau$-invariant, and by Corollary \ref{cor:BlowupandAdinvarianceFutaCar}, the Futaki character of $X$ is identically zero on its K\"ahler cone.

\subsection{Family 4.7}
Let $X$ be a smooth Fano threefold in family 4.7. Then it is a blow-up of a smooth divisor $W$ of bidegree $(1,1)$ on $\P^2\times \P^2$ along two disjoints curves of bidegrees $(1,0)$ and $(0,1)$, and it is $K$-polystable \cite{SussKE,SussPictureBook,IltenSuss}. To perform computations, we will assume that 
$$
W=\lbrace xu+yv+zw = 0 \rbrace\subset \P^2\times\P^2,
$$
where $[x,y,z]$ and $[u,v,w]$ stand for homogeneous coordinates on the first and second factors respectively. We will denote by $\pi_i : W \to \P^2$ the natural projection on the $i$-th factor. We then let $C_i=\pi_i^{-1}([0:0:1])\subset W$. Then, $X=\Bl_{C_1\cup C_2}(W)$ and from \cite[Lemmas 7.1 and 7.7]{cps19}, we have
$$
\Aut_0(X)\simeq \GL_2(\C).
$$
The isomorphism is defined as follows. First, automorphisms of $X$ are induced by automorphisms of $W$ that leave $C_1\cup C_2$ invariant. Arguing as in \cite[Lemma 7.7]{cps19}, they correspond to lift of isomorphisms of $\P^2$ that leave the set
$$\pi_1(C_1\cup C_2)=\lbrace [0:0:1]\rbrace \cup\lbrace [x: y : 0], (x,y)\in \C^2\setminus \lbrace 0\rbrace \rbrace $$ 
invariant. Those elements are easily identified to elements in $\GL_2(\C)$. From Section \ref{sec:sln}, the Futaki invariant vanishes on the $\sl_2(\C)$-component in $\aut(X)$. We can identify a supplementary subspace of $\sl_2(\C)$ in $\aut(X)$ by considering the lift to $X$ of the generators of the $\C^*$-action on $\P^2$ given by
$$
\lambda\cdot ([x:y:z])=([\lambda x : y: z]).
$$
The lift of this action to $W$ is given by
\begin{equation}
 \label{eq:action4.7}
\lambda\cdot ([x:y:z],[u:v:w])=([\lambda x:y:z],[\lambda^{-1}u:v:w]).
\end{equation}
We introduce the involution
$$
\tau([x:y:z],[u:v:w])=([u:v:w],[x:y:z]).
$$
This preserves $W$, and swaps the curves $C_1$ and $C_2$. It also swaps the $(1,1)$-classes $\pi_1^*[\om_{FS}]$ and 
$\pi_2^*[\om_{FS}]$. Finally, its adjoint actions maps a generator of the $\C^*$-action (\ref{eq:action4.7}) to its inverse. Then, following Section \ref{sec:prelim}, we deduce the vanishing of the Futaki invariant on $X$ for any K\"ahler class of the form
$$
c_1(X)+\ep \pi^*(\pi_1^*[\om_{FS}]+\pi_2^*[\om_{FS}])+\eta (c_1(\cO(E_1))+c_1(\cO(E_2))),
$$
where $\pi : X \to W$ denotes the blow-down map, $E_1$ and $E_2$ the exceptional divisors, and $(\ep,\eta)\in\R^2$ are chosen so that the class is positive.

\section{Remaining cases}
We finish with families N\textdegree $\cN$, with $$\cN\in\lbrace 2.20, 3.9, 3.13, 4.2  \rbrace.$$

\subsection{Family 2.20}

Consider the Pl\"ucker embedding of $\Gr(2,5)$ in $\P^9$. Any smooth intersection of this embedded sixfold with a linear subspace of codimension 3 is a Fano manifold. We call this Fano threefold $V_5$ and it is the unique member of family 1.15 of Fano threefolds.

Now, let $C$ be a twisted cubic in $V_5$ and let $X=\Bl_C( V_5)$. Then $X$ is a member of the family 2.20 of Fano threefolds. Up to isomorphism, there is a unique choice of curve such that $X$ has infinite automorphism group \cite[Lemma 6.10]{cps19}. In this case, $\Aut(X)$ is a semidirect product $\C^*\rtimes\Z/2\Z$.

In \cite[Section 5.8]{fanothreefolds}, it is shown that the unique element in family 2.20 with infinite automorphism group is $K$-polystable. Moreover, the following explicit description of $X$ is given. First, $V_5$ can be realised as the subvariety of $\P^6$ cut out by the equations 
$$
\left\{
\begin{array}{lllll}
x_4x_5-x_0x_2+x_1^2 = 0 \\
x_4x_6-x_1x_3+x_2^2=0 \\
x_4^2 - x_0x_3+x_1x_2 = 0 \\
x_1x_4 -x_0x_6 -x_2x_5 = 0 \\
x_2x_4 -x_3x_5 -x_1x_6 = 0.
\end{array}
\right.
$$
We will then identify $V_5$ with this variety. Then, we can chose $C$ to be the twisted cubic parametrised by 
$$
([r:s]) \mapsto ([r^3 : r^2s:rs^2 : s^3 : 0 : 0 : 0])\in V_5.
$$
We consider $X=\Bl_C(V_5)$ with this parametrisation. The  $\C^*\rtimes\Z/2\Z$-action on $\P^6$ generated by 
$$
\lambda\cdot([x_0:x_1:x_2:x_3:x_4:x_5:x_6])= 
[\lambda^3x_0:\lambda^5x_1:\lambda^7x_2:\lambda^9x_3:\lambda^6x_4:\lambda^4x_5:\lambda^8 x_6]
$$
for $\lambda \in \C^*$ and the involution
$$
\tau([x_0:x_1:x_2:x_3:x_4:x_5:x_6])=[x_3:x_2:x_1:x_0:x_4:x_6:x_5]
$$
preserves $V_5$ and $C$, hence lifts to $X$. This provides the isomorphism $$\Aut(X)\simeq \C^*\rtimes\Z/2\Z.$$
Note that conjugation by $\tau$ sends a generator of the $\C^*$-action to its inverse. As $H^{1,1}(V_5,\R)$ is generated by the class of a hyperplane section in $\P^6$, and as the class of the Fubini--Study metric on $\P^6$ is $\tau$-invariant, we can apply Corollary \ref{cor:BlowupandAdinvarianceFutaCar} to $X$, and we deduce the vanishing of the Futaki invariant on the K\"ahler cone of $X$.

\subsection{Families 3.9 and 4.2}
We now consider families 3.9 and 4.2. Any member of one of these families have $\Aut_0(X)\simeq \C^*$ and is $K$-polystable by \cite[Section 4.6]{fanothreefolds}, which we will follow closely. 

Let 
$\cS$ be either $\P^2$ or $\P^1\times \P^1$ and let $\cC\subset \cS$ be a smooth irreducible curve given by a quartic if $\cS=\P^2$ and a $(2,2)$-curve in the other case. Denote by $\mathrm{pr}_i$ the projection of $\P^1\times \cS$ onto the $i$-th factor. We then set $\cB=\mathrm{pr}_2^*(\cC)\simeq \P^1 \times \cC$, $\cE=\mathrm{pr}_1^*([1:0])$ and $\cE'=\mathrm{pr}_1^*([0:1])$. We consider 
$$G=\C^*\rtimes\Z/2\Z$$
acting on $\P^1$ by 
$$\lambda\cdot[u:v]=[u:\lambda v]$$
and 
$$\tau([u:v])=[v:u].$$
The $G$-action lifts to $\P^1\times \cS$, with the involution $\tau$ swapping $\cE$ and $\cE'$. We then introduce $\eta : W \to \P^1\times \cS$ a double cover branched over $\cE+\cE'+\cB$, and $\overline{E}$, $\overline{E}'$ and $\overline{B}$ the preimages on $W$ of the surfaces $\cE$, $\cE'$ and $\cB$ respectively. Then, set $\hat X \to W$ the blow-up of $W$ along the curves $\overline{E}\cap\overline{B}$ and $\overline{E}'\cap\overline{B}$ with exceptional surfaces $\hat S$ and $\hat S'$. We denote the proper transforms of $\overline{E}$, $\overline{E}'$ and $\overline{B}$  by $\hat E$, $\hat E'$, $\hat B$  respectively.  Finally, $X$ is obtained as the image of a contraction $\hat X \to X$ of $\hat B$ to a curve isomorphic to $\cC$. We set $E$, $E'$, $S$ and $S'$ the proper transforms on $X$ of $\hat E$, $\hat E'$, $\hat S$ and $\hat S'$ respectively. 

One can check that all the birational maps involved in producing $X$ are $G$-equivariant, and we obtain $\Aut_0(X)\simeq \C^*$. Moreover, the involution on $X$ induced by $\tau$ (that we will still denote $\tau$) swaps $E$ and $E'$, and also swaps $S$ and $S'$. Hence, the K\"ahler classes $c_1(E)+c_1(E')$ and $c_1(S)+c_1(S')$ are both $\tau$-invariant. Clearly, on $\P^1\times \cS$, the adjoint action of the involution $\tau$ maps a generator of the $\C^*$-action to its inverse. This remains true on  $W$ by equivariance, and thus on $X$ that is birationally equivalent to $W$ by continuity of holomorphic vector fields away from the exceptional loci. Then, Proposition \ref{prop:AdinvarianceFutCar} applies to show that the Futaki invariant of $X$ vanishes in any K\"ahler class of the form 
$$
c_1(X)+\ep(c_1(E)+c_1(E'))+\delta(c_1(S)+c_1(S')),
$$
where $(\ep,\delta)\in\R^2$ is chosen so that the class is positive. 

To understand the subset of the K\"ahler cone these classes generate, we use the following alternative description of $X$, still following \cite[Section 4.6]{fanothreefolds}. 
\subsubsection{Family 3.9}
This is the case when $\cS=\P^2$. $X$ can then also be obtained as the blow-up $\phi : X \to V$ of $V$ along a curve $C\subset V$ where 
$$\pi : V=\P(\cO\oplus \cO(2))\to \P^2=\cS$$
is a $\P^1$-bundle, and $C=\pi^*\cC\cap E_V$, where $E_V$ is the zero section of $\pi$. We also have that the strict transform of $E_V$ (resp. of the infinity section $E'_V$, and of $\pi^*\cC$) on $X$ is $E$ (resp. $E'$ and $S'$), while the exceptional divisor of $\phi$ is $S$. Hence we get the relation $c_1(E)+c_1(E')=0$ in this case (but $c_1(S)+c_1(S')\neq 0$), and we obtain a $2$-dimensional family of classes that admit cscK metrics given by 
$$
(\delta,r) \to r(c_1(X)+\delta(c_1(S)+c_1(S'))).
$$
\subsubsection{Family 4.2}
This is the case when $\cS=\P^1\times\P^1$. Again, we can recover $X$ from the maps 
$$
\pi : X \to V
$$
and 
$$
\phi : V \to \cS,
$$
with $\pi$ the contraction of $S$ to a curve isomorphic to $\cC$ and $\phi$ a $\P^1$-bundle over $\P^1\times \P^1$. According to \cite[Section 10]{Fujita2016}, we have
$$
\mathrm{Pic}(X)=\Z[H_1]\oplus\Z[H_2]\oplus \Z[E]\oplus\Z[E']
$$
where $H_i=(\pi\circ\phi)^*(\ell_i)$ and $\ell_1, \ell_2$ denote two different rulings of $\P^1\times\P^1$. There are relations $ -K_X\sim 2(H_1+H_2)+E+E'$, $S\sim H_1+H_2 -E+E'$ and $S'\sim H_1+H_2 +E-E'$, so that the K\"ahler classes described above can be written
$$
2(1+\delta)(c_1(H_1)+c_1(H_2))+(1+\ep)(c_1(E)+c_1(E')).
$$
Together with scaling we therefore obtain a $3$-dimensional family of classes with vanishing Futaki invariant.

\subsection{Family 3.13}
Let $X$ be a smooth $K$-polystable Fano threefold in family 3.13. From \cite[Section 5.19]{fanothreefolds}, 
either  $\Aut_0(X)\simeq \mathrm{PGL}_2(\C)$, and so from Section \ref{sec:sln} the Futaki invariant vanishes identically, or $\Aut_0(X)\simeq\C^*$. In the latter case, denoting $[x_0:x_1:x_2]$, $[y_0:y_1:y_2]$ and $[z_0:z_1:z_3]$ the homogeneous coordinates on the first, second and third factors of $\P^2\times\P^2\times\P^2$, $X$ is given by the equations 
$$
\left\{
\begin{array}{ccc}
 x_0y_0+x_1y_1+x_2y_2 & = & 0\\
 y_0z_0+y_1z_1+y_2z_2 & = &0 \\
 (1 + s)x_0 z_1 + (1 -s)x_1 z_0 -2x_2 z_2 & = & 0
\end{array}
\right.
$$
in $\P^2\times\P^2\times\P^2$, for $s\notin\lbrace -1,0,1 \rbrace$, and
$$
\Aut(X)\simeq\C^*\rtimes\mathfrak{S}_3.
$$
The $\C^*$-action for $\lambda\in\C^*$ is given on a point $P$ with homogeneous coordinates $([x_0:x_1:x_2], [y_0:y_1:y_2],[z_0:z_1:z_3])$ by 
$$\lambda\cdot(P)=
([\lambda x_0:\lambda^{-1}x_1:x_2], [\lambda^{-1}y_0:\lambda y_1:y_2],[\lambda z_0:\lambda^{-1}z_1:z_3]).
$$
Further, there are two involutions $\tau_{x,z}$ and $\tau_{y,z}$ in $\Aut(X)$, whose actions are given by 
$$
\tau_{x,z}(P)=([z_1 : z_0 : z_2 ], [y_1 : y_0 : y_2 ], [x_1 : x_0 : x_2 ])
$$
and 
$$
\tau_{y,z}(P)=([x_1 : x_0 : -x_2],[(1-s)z_0:(1+s)z_1:2z_2],[\frac{y_0}{1-s} :\frac{y_1}{1+s} : \frac{y_2}{2}]).
$$

Note that $\tau_{x,z}\circ\lambda\circ\tau_{x,z}^{-1}=\lambda^{-1}$ and $\tau_{y,z}\circ\lambda\circ\tau_{y,z}^{-1}=\lambda^{-1}$ (where we identified $\lambda$ with the corresponding element in $\Aut(X)$). From \cite[Diagram 5.19.1]{fanothreefolds}, the projection maps $\eta_x,\eta_y,\eta_z : \P^2\times\P^2\times\P^2 \to \P^2$ induce holomorphic maps, still denoted $\eta_x, \eta_y$ and $\eta_z$, from $X$ to $\P^2$.
If we denote $\alpha_i:= \eta_i^*[\om_{FS}]\in H^{1,1}(X,\R)$ the pullback of the class of the Fubini--Study form, for $i\in\lbrace x,y,z\rbrace$, by equivariance of the projections, we see that $\alpha_y$ is $\tau_{x,z}$-invariant while $\alpha_x$ is $\tau_{y,z}$-invariant. Hence, for any $\ep>0$ small enough, the class $c_1(X)+\ep \alpha_x$ is $\tau_{y,z}$-invariant and the class $c_1(X)+\ep\alpha_y$ is $\tau_{x,z}$-invariant. From Proposition \ref{prop:AdinvarianceFutCar}, the Futaki invariants of $(X,c_1(X)+\ep \alpha_x)$ and $(X,c_1(X)+\ep \alpha_y)$ vanish. Hence, $X$ will carry cscK deformations of its K\"ahler--Einstein metrics in the classes $c_1(X)+\ep \alpha_y$ and $c_1(X)+\ep \alpha_x$ for $\ep$ small enough by LeBrun--Simanca's openness theorem.
\begin{remark}
We have used two different involutions $\tau_{x,z}$ and $\tau_{y,z}$ in the above to the deduce the vanishing of the Futaki invariant in the classes  $c_1(X)+\ep \alpha_y$ and $c_1(X)+\ep \alpha_x$. We are therefore not able from these arguments to deduce that the Futaki invariant vanishes on the sums of these classes. Hence we still only get a $2$-dimensional family of classes with vanishing Futaki invariant in this case.
\end{remark}

\bibliography{biblib}
\bibliographystyle{alpha}

\end{document}